\documentclass[preprint, 12pt, authoryear]{amsart}
\usepackage{amssymb, amsthm, amsmath, euler}
\usepackage{latexsym}

\theoremstyle{plain}
\newtheorem{theorem}{Theorem}
\newtheorem{coro}{Corollary}
\newtheorem{lemma}{Lemma}

\newtheorem{rmk}{Remark}

\theoremstyle{definition}
\newtheorem{definition}{Definition}
\newtheorem{example}{Example}

\numberwithin{theorem}{section}
\numberwithin{coro}{section}
\numberwithin{lemma}{section}
\numberwithin{proposition}{section}
\numberwithin{rmk}{section}
\numberwithin{definition}{section}
\numberwithin{example}{section}
\numberwithin{que}{section}

\newcommand{\CC}{\mathbb{C}_\infty} 
\newcommand{\Fq}{\mathbb{F}_q} 
\renewcommand{\AA}{A} 

\newcommand{\ppi}{\widetilde{\pi}}

\newcommand{\GL}{\text{GL}_2(\AA)}
\newcommand{\HHH}{\Omega}

\renewcommand{\gg}{g}

\renewcommand{\slash}{{\mid_{[\gamma]_{k, m}}}}

\DeclareMathOperator{\ddet}{det}

\DeclareMathOperator{\DDd}{D}
\DeclareMathOperator{\Dunalt}{\mathscr{D}}

\newcommand{\jg}{j_\gamma(z)}
\newcommand{\Xg}{X_\gamma(z)}

\newcommand{\val}{\mathtt{val}}

\newcommand{\ord}{\mathtt{ord}}

\newcommand{\degg}{deg}

\begin{document}


\title{On hyperderivatives of single-cuspidal Drinfeld modular forms with $A$-expansions}

\author{Aleksandar Petrov}
\address{Max Planck Institute for Mathematics \\
 Vivatsgasse 7 \\
  53111 Bonn, Germany}


\maketitle

\begin{abstract}
We show that the Drinfeld modular forms with $A$-expansions that have been constructed by the author are precisely the hyperderivatives of the subfamily of single-cuspidal Drinfeld modular forms with $A$-expansions that remain modular after hyperdifferentiation. In addition, we show that certain Drinfeld-Poincar\'{e} series display a similar behavior with respect to hyperdifferentiation, giving indirect evidence that the Drinfeld modular forms with $A$-expansions are Drinfeld-Poincar\'{e} series. The Drinfeld-Poincar\'{e} series that we consider generalize previous examples of such series by Gekeler, and Gerritzen and van der Put. 
\end{abstract}



\section{Introduction and notation}

The purpose of this short note is to prove that the Drinfeld modular forms that have been constructed in \cite[Thm.1.3]{Pet} are precisely the hyperderivatives of the subfamily of single-cuspidal Drinfeld modular forms with $A$-expansions that remain modular after differentiation. The result is a consequence of results scattered in \cite{BoPe} and \cite{Pet}, but has not appeared in a single and concise form. 

In their paper \cite{BoPe} Bosser and Pellarin show that hyperderivatives preserve $A$-expansions, but that at the same time hyperderivatives do not preserve modularity in general. In \cite{Pet} the author constructs an infinite family of Drinfeld modular forms with $A$-expansions and notes that the subfamily of single-cuspidal Drinfeld modular forms with $A$-expansions is special. For example, the subfamily has an element in each possible weight in which a single-cuspidal form can appear. We connect the two results by checking (with the help of Lucas' Theorem) that every form with an $A$-expansion constructed by the author can be obtained from a single-cuspidal form with an $A$-expansion by using hyperdifferentiation.

We note that this gives an alternative proof of the main result in \cite{Pet}. Namely, one can first prove the existence of the special subfamily of single-cuspidal forms with $A$-expansions and then apply hyperderivatives to get the rest. However, proving the existence of the special family of single-cuspidal forms appears to be just as difficult as proving the existence of the whole family of forms with $A$-expansions. 

Finally, we show that when hyperdifferentiation preserves modularity, the Drinfeld-Poincar\'{e} series of type $1$ hyperdifferentiate to new Drinfeld-Poincar\'{e} series. By combining the results of Gekeler (\cite[Sec.~9]{Gek}) and Lopez \cite{Lop} one can show that one such series is the negative of the single-cuspidal form with $A$-expansion of smallest weight. The observed similarity under hyperdifferentiation suggests that Drinfeld modular forms with $A$-expansions are Drinfeld-Poincar\'{e} series.

\

\emph{Acknowledgement:} Part of the present work is contained in the author's Ph.D. thesis. The author wants to thank Dinesh Thakur, who suggested the original project, and David Goss, who persuaded the author that writing this short note is a useful endeavor. Finally, the author wants to thank the anonymous referee for the detailed report on the original version of this note, which helped to greatly improve the article.

\

We use the following standard notation. Let $p$ be a prime number, and let $q = p^l$ for a fixed positive integer $l$. Let $\Fq$ be the field with $q$ elements, $A = \Fq[T]$, $K = \Fq(T)$, $A_+$ the set of monic elements in $A$. We denote by $| \cdot |$ the absolute value on $K$ uniquely defined by $|a| = q^{\degg_T (a)}$, for $a \in A$. Let $K_\infty$ be the completion of $K$ w.r.t. $|\cdot|$, $\mathbb{C}_\infty$ be the completion of a fixed algebraic closure of $K_\infty$. Let $\Omega = \mathbb{C}_\infty \setminus K_\infty$ be the Drinfeld upper half-plane (see \cite[Prop.1.65]{Gos_eisenstein}) and $\Gamma = \GL$. Let $M_{k, m}$ be the finite-dimensional $\mathbb{C}_\infty$-vector space of Drinfeld modular forms for $\GL$ of weight~$k$ and type~$m$; let $S_{k, m}$ be the subspace of cuspidal Drinfeld modular forms inside $M_{k, m}$ (see \cite{Gek}, \cite{Goss_mod_forms} and \cite{Gos_eisenstein} for more details). We set
\[
	u(z) := \frac{1}{\ppi} \sum_{a \in A} \frac{1}{z+a},
\]
where $\ppi$ is a fixed choice of fundamental period of the Carlitz module. The function $u = u(z)$ is the usual uniformizer at 'infinity', which is denoted by~$t$ in \cite{Gek} and \cite{Pet}. One can show that for $a \in A_+$, $u_a := u(az)$ is a power series in $u$ with coefficients in $A$, which starts with $u^{|a|}$ (see \cite[(6.2)]{Gek}). If $n$ is a non-negative integer, then we let $G_n (X)$ be the $n$-th Goss polynomial for the lattice $\ppi A$, which corresponds to the Carlitz module (see \cite[(4.1)]{Gek}). The polynomial $G_n$ is monic of degree $n$ with coefficients in $K$ (see \cite[(3.4) \& (4.1)]{Gek}). In fact, the explicit description of the coefficients of the Carlitz exponential (\cite[(4.3)]{Gek}) together with formula \cite[(3.8)]{Gek} show that the non-zero coefficients of $G_n$ have absolute value $\leq 1$, which will be useful for us later. The reader should note that by definition of $G_n$ we have
\[
	G_n (u(z)) = \frac{1}{\ppi^n} \sum_{a \in A} \frac{1}{(z+a)^n}.
\]


In addition, if 
\[
	\gamma = \begin{bmatrix} a & b \\ c & d \end{bmatrix} \in \Gamma,
\]
then we define
\[
	\jg := c z + d, \qquad \qquad \Xg := \frac{c}{c z + d}.
\]

\section{Single-cuspidal modular forms with $A$-expansions}

In this section, we review the theory of single-cuspidal Drinfeld modular forms with $A$-expansions from \cite{Pet}.

Recall that for $z \in \Omega$ we have a non-archimedean analog of `imaginary distance': $|z|_i = \inf_{k \in K_\infty} | z - k|$. If $f \in M_{k, m}$, then for $|z|_i$ sufficiently large:
\[
	f(z)  = \sum_{i= 0}^\infty a_i u(z)^i \in \mathbb{C}_\infty [[u(z)]].
\]
The series on the right is called the expansion of $f$ at `infinity'. It was discovered by Goss  \cite[1.77]{Gos_eisenstein} and can be viewed as the analog of the Fourier expansion of a classical modular forms. Such an expansion determines $f$ uniquely because $\Omega$ is a connected rigid-analytic space. The main objects of interest in this paper are Drinfeld modular forms that have expansions of a different kind:


\begin{definition} A modular form $f \in S_{k, m}$ is said to have an \emph{$A$-expansion of exponent~$n$} if there exist both a positive integer $n$ and coefficients $c_a \in \mathbb{C}_\infty$ such that
\[
	f = \sum_{a \in A_+} c_a G_n (u_a),
\]
where the equality is meant as an equality in $\mathbb{C}_\infty[[u]]$ between the expansion at `infinity' of $f$ on the left side and the expression on the right side.
\end{definition}


\begin{rmk}
One can show that $m - n \equiv 0 \bmod (q-1)$, and that if we fix the exponent $n$, then the coefficients $\{ c_a \}_{a \in A_+}$ determine $f$ uniquely.  
\end{rmk}

We have the following result concerning the existence of $A$-expansions:

\begin{theorem} \label{special_family}
Let $s$ be a positive integer. Then
\[
	f_s : = \sum_{a \in A_+} a^{1 + s(q-1)} u_a
\]
is a single-cuspidal Drinfeld modular form of weight $2 + s (q-1)$ and type~$1$, i.e., $f_s \in S_{2 + s(q-1), 1}$ and single-cuspidal.  
\end{theorem}

\begin{proof}
This is a special case of \cite[Thm.1.3]{Pet}. We note that the proof \cite[Sec.4]{Pet} of the special case is no simpler than the proof of the general case.
\end{proof}


\begin{rmk}
We have changed the indexing from \cite{Pet} and we now start with $s = 1$, rather than with $s = 0$ as in \cite[Def.3.1]{Pet}.
\end{rmk}

\begin{rmk} \label{eigen_prop}
The existence of an $A$-expansion for $f_s$ of exponent $1$ easily implies  (see \cite[Thm.2.3]{Pet}) that $f_s$ is a simultaneous Hecke eigenform with eigensystem $\{ \wp \}$. 
\end{rmk}

Arguably the most important member of the family $\{ f_s \}$ is the form $f_1 = h$, whose $A$-expansion was discovered by Lopez in \cite{Lop}:
\[
	h = \sum_{a \in A_+} a^q u_a.
\]

\section{Hyperderivatives of the family $\{ f_s \}$}

In this section, we recall the theory of hyperderivatives\footnote{Some authors also call these Hasse derivative.} applied to Drinfeld modular forms (see \cite[Sec.3]{BoPe}). 

\begin{definition}
 Let $f: \HHH \to \CC$ be a rigid analytic function. Define its $n^\text{th}$ \emph{unaltered hyperderivative}, $\Dunalt_n f(z)$, by the formula
\[
 f (z + \varepsilon) = \sum_{ n \geq 0} \Dunalt_n f (z) \varepsilon^n,
\]
where $\varepsilon$ is small in absolute value.  The $n^\text{th}$ \emph{hyperderivative}, $\DDd_n f$, is defined by the formula
\[
  \DDd_n = \frac{(-1)^n}{\ppi^n} \Dunalt_n.
\]
\end{definition}

One can show (see \cite[Sec.3]{BoPe} and the references therein) that the $\DDd_n$'s (and the $\Dunalt_n$'s) form an iterative family of higher differentials.  In particular, we have a Leibniz rule:
\begin{equation} \label{leibniz}
\Dunalt_n (f g) = \sum_{ r = 0}^n \Dunalt_r (f) \Dunalt_{n-r} (g).
\end{equation}

It is an easy exercise to show that
\[
 \DDd_n \left ( \frac{1}{(\ppi a z + \ppi b)^w} \right ) = \binom{w + n -1}{n}  \frac{a^n}{(\ppi a z + \ppi b)^{n + w}}
\]
and therefore the hyperderivatives $\{ \DDd_n \}$ preserve $A$-expansions:
\[
	\DDd_n \left (\sum_{a \in \AA_+} c_a G_w (u_a) \right) = \binom{w + n -1}{n} \sum_{a \in \AA_+} c_a a^n G_{n + w} (u_a).
\]
In particular for $w = 1$, we see that
\begin{equation} \label{hasse_action}
 \DDd_n \left ( \sum_{a \in \AA_+} c_a u_a \right ) =  \sum_{a \in \AA_+} c_a a^n G_{n+1} (u_a).
  \end{equation}
  
Next we consider whether hyperderivatives preserve modularity. Bosser and Pellarin have shown  \cite[Lem.3.4]{BoPe} the following result concerning the action of $\Dunalt_n$ on Drinfeld modular forms:

\begin{theorem} \label{Bose}
 If $f \in M_{k, m}$, then for $\gamma \in \Gamma$ we have
\[
  \begin{aligned}
   & \Dunalt_n f (\gamma z)  =  (\ddet \gamma)^{-m-n} \jg^{k + 2n} \Dunalt_n f(z)  \\
  &  \qquad \qquad \qquad + (\ddet \gamma)^{-m-n} \jg^{k + 2n} \sum_{j = 1}^n \binom{k + n -1}{j} \Xg^j \Dunalt_{n-j} f(z). 
  \end{aligned}
\]
\end{theorem}

\begin{rmk}  \label{rmk_imp} In particular, Theorem~\ref{Bose} shows that if
\[
  \binom{k + n -1}{j} \equiv 0 \bmod p, \qquad \forall j, \ 1 \leq j \leq n,
\]
then $\Dunalt_n$ (and $\DDd_n$) preserves modularity, i.e., $\Dunalt_n, \DDd_n: M_{k, m} \to M_{k + 2n, m + n}$.
\end{rmk}

We want to use the hyperderivatives $\DDd_n$ and the special family $\{ f_s \}$ to produce other forms with $\AA$-expansions as in the following definition.

\begin{definition} \label{f_k_m} Let $k$, $n$ be positive integers, we define the power series $f_{k, n}$ by the formula
\[
	f_{k, n} := \sum_{a \in A_+} a^{k-n} G_n (u_a) \in K [[u]].
\]
\end{definition}

The reader should note that in this notation $f_s = f_{2 + s(q-1), 1}$.

We have the following result:

\begin{theorem} \label{main}
Assume that $k, n$ are positive integers such that $k - 2 n$ is a positive multiple of $(q-1)$ and $n \leq p^{\val_p (k - n)}$. If \[ s = \frac{k - 2 n}{q-1} , \] then
\[
	\DDd_{n - 1} f_s = f_{k, n}
\]	
is an element of $S_{k, n}$.
\end{theorem}

\begin{proof}
One immediately checks that \eqref{hasse_action} gives
\[
	\DDd_{n - 1} \left ( \sum_{a \in A_+} a^{1 + s (q-1)} u_a  \right) = \sum_{a \in A_+} a^{k-n} G_n (u_a).
\]
It remains to show that this is modular, which can be accomplished by showing that $\forall j, 1 \leq j \leq n - 1$, we have
\begin{equation} \label{eq_in_thm}
	\binom{2 + s(q-1) + (n - 1) -1}{j} = \binom{k - n}{j} \equiv 0 \bmod p.
\end{equation}
Recall that Lucas' Theorem says that if
\[
\begin{aligned}
	\sigma &  = \sigma_0 + \sigma_1 p + \cdots + \sigma_r p^r + \sigma^\ast p^{r+1},  &  \qquad 0 \leq \sigma_i < p, \\
	j &  = j_0 + j_1 p + \cdots + j_r p^r, & 0 \leq j_i < p, \\
\end{aligned}
\]
then
\[
	\binom{\sum_i \sigma_i p^i}{\sum_i j_i p^i} \equiv \binom{\sigma_0}{j_0} \binom{\sigma_1}{j_1} \cdots \binom{\sigma_r}{j_r} \bmod p.
\]
Combining Lucas' Theorem with $n \leq p^{\val_p (k - n)}$ shows that \eqref{eq_in_thm} holds for $1 \leq j \leq n - 1$, completing the proof.
\end{proof}

\begin{rmk}
As we have already remarked \cite[Thm.1.3]{Pet} proves that $f_{k, n}$ are elements of $S_{k, n}$ when $k, n$ are positive integers such that $k - 2 n$ is a positive multiple of $(q-1)$ and $n \leq p^{\val_p (k - n)}$. Theorem~\ref{main} shows that every $f_{k, n}$ is a hyperderivative of some form $f_s$ which is modular. Therefore the forms that occur in \cite[Thm.1.3]{Pet} are precisely the hyperderivatives of the family $\{ f_s \}$ that remain modular.
\end{rmk}

\begin{rmk}
The main property of the $A$-expansions $f_{k, n}$ that we were interested in \cite{Pet} is that they are Hecke eigenforms with eigensystems $\{ \wp^n \}$. Lemma~4.6 from \cite{BP} shows that we recover this from the result quoted in Remark~\ref{eigen_prop}. 
\end{rmk}

Theorem \ref{Bose} shows that in general the hyperderivatives of Drinfeld modular forms are Drinfeld quasi-modular forms (see \cite[Def.2.1]{BoPe} for the definition). Therefore without any conditions on $n$, we have that $\DDd_n f_s$ is a Drinfeld quasi-modular form with an $A$-expansion, which in addition is a Hecke eigenform. This provides new examples of quasi-modular forms with non-zero expansions at `infinity'. At present the author does not know how this contributes to the theory of quasi-modular forms of Bosser and Pellarin in other ways.

Bosser and Pellarin also consider a different kind of higher derivative, namely Serre derivatives, of Drinfeld modular forms in \cite[Sec.4.1]{BP}. Serre derivatives have the advantage of preserving modularity, but they do not preserve $A$-expansions and Hecke properties in general.

We end this section with an example.

\begin{example}
Let $q = 3$. The following is a complete list of $\DDd_n h$ that are modular for $n \leq 1000$:
\[
\begin{aligned}
\DDd_6 h & = \sum_{a \in A_+} a^{9} G_7 (u_a) \in S_{16, 1}, \\
\DDd_{24} h & = \sum_{a \in A_+} a^{27} G_{25} (u_a) \in S_{52, 1}, \\
\DDd_{78} h & = \sum_{a \in A_+} a^{81} G_{79} (u_a) \in S_{160, 1}, \\
\DDd_{240} h & = \sum_{a \in A_+} a^{243} G_{241} (u_a) \in S_{484, 1}, \\
\DDd_{726} h & = \sum_{a \in A_+} a^{729} G_{727} (u_a) \in S_{1456, 1}.
\end{aligned}
\]
\end{example}

\section{Hyperderivatives of Drinfeld-Poincar\'{e} series}

In this section, we define Drinfeld-Poincar\'{e} series and we show how some of them behave under hyperdifferentiation.

If $f : \Omega \to \CC$ is a rigid analytic function, then define the \emph{slash operator of weight $k$ and type $m$}, $\slash$,  to be
\[
	\left ( f \slash \right ) (z) = \frac{(\det \gamma)^m}{(c z + d)^k} f(\gamma z) = \frac{(\det \gamma)^m}{\jg^k} f(\gamma z), \qquad \forall \gamma \in \Gamma.
\] 

Let $H$ be the subgroup 
\[
	\left \{ \begin{bmatrix} \ast & \ast \\ 0 & 1 \end{bmatrix} \right \} \subset \Gamma.
\]
We have $H \backslash \Gamma \cong \{ (c, d) : c, d \in A, (c, d) = 1 \}$ (see \cite[(5.11)]{Gek}) under the map
\[
	\gamma_{c, d} : =  \begin{bmatrix} \ast & \ast \\ c & d \end{bmatrix}  \mapsto (c, d).
\]
We will use this identification below.

\begin{definition} Let $k, n$ be two non-negative integers and let $m$ be the reduction of $n$ modulo $(q-1)$. Define the \emph{Drinfeld-Poincar\'{e} series of weight $k$ and type $m$}, $P_{k, n}$,  to be
\begin{equation} \label{Poincare_series}
P_{k, n}  = \sum_{\gamma \in H \backslash \Gamma} G_n (u) \slash = \sum_{\gamma \in H \backslash \Gamma} \frac{(\det \gamma)^m}{(cz + d)^k} G_n (u(\gamma z)).
\end{equation}
\end{definition}

The previous definition is similar to the definition of Poincar\'{e} series by Gekeler in \cite[(5.11)]{Gek} and to the definition of Eisenstein series for $\text{PSL}_2 (A)$ by Gerritzen and van der Put in \cite[p.304]{GvP}. In particular, our definition coincides with Gekeler's definition if $G_n (X) = X^n$ (eg. $n \leq q$).

\begin{theorem}
For $k, n, m$ as in the definition the function $P_{k, n}$ is a Drinfeld modular form of weight $k$ and type $m$. Further, $P_{k, n}$ is non-zero if $k \equiv 2 n \bmod (q-1)$ and $n \leq \frac{k}{q+1}$.
\end{theorem}

\begin{proof}
Gerritzen and van der Put prove (\cite[p.304-305]{GvP}) that series of the form
\[
	 \sum_{\gamma \in H \backslash \Gamma} \frac{1}{(cz + d)^k} u^w(\gamma z), \qquad w > 0,
\]
are Drinfeld modular forms  of weight $k$ and type $w$. Since $G_n(u)$ is a polynomial in $u$, it is then immediate that $P_{k, n}$ is a Drinfeld modular form of weight $k$ and type $m$.

For the second statement we again take the approach of Gerritzen and van der Put \cite[p.305-307]{GvP}. We fix a positive integer $n$ such that $n \leq \frac{k}{q+1}$. Write
\[
	G_n (X) = X^n + \gg_1 X^{n - (q-1)} + \cdots + \gg_s X^{n - s (q-1)}, 
\]
with $\ord_X G_n = n - s(q-1)$. Recall that $|\gg_i| \leq 1$ and $\gg_i \in K$. For convenience we also let $\gg_0 = 1$.

Let $\xi_N = \sqrt[N]{T}$. We will show that if $N > n q (q-1)$, then $P_{k, n} (\xi_N) \neq 0$. 

We split the series defining $P_{k, n}$ into three parts. 

First, consider the part corresponding to matrices $\gamma_{c, d} \in H \backslash \Gamma$ with $\gamma_{c, d} =  \begin{bmatrix} d^{-1}& 0 \\ 0 & d \end{bmatrix}$, $d \in \Fq^\times$:
\[
	S_1 = \sum_{c = 0, d \in \Fq^\times} \frac{1}{d^{k}} G_n \left (u \left ( \frac{d^{-1} \xi_N}{d} \right) \right).
\]
As shown in \cite[p.305]{GvP} we get $S_1 = -G_n (u (\xi_N))$. By using the argument in \cite[Lem.5.5]{Gek} we have \[ |u(\xi_N)| = |\ppi^{-1}| |\xi_N|^q = q^{\frac{-q}{q-1}}  q^{\frac{q}{N}}. \] Combining this with $N > n q (q-1)$ and $|\gg_i| \leq 1$, one easily checks that for $0 \leq i, j \leq s$, $i \neq j$ we have
\[
	|\gg_i u (\xi_N)^{n - i (q-1)} | \neq |\gg_j u (\xi_N)^{n - j(q-1)}|.
\]
Hence 
\begin{equation} \label{s1}
	|S_1| = \max_{0 \leq i \leq s} |\gg_i| |u (\xi_N)|^{n - i (q-1)} = \max_{0 \leq i \leq s} |\gg_i| \frac{1}{|\ppi|^{n - i (q-1)}} \frac{1}{|\xi_N|^{q (n - i (q-1))}}.
\end{equation}
Second, consider the part corresponding to matrices $\gamma_{c, d} \in H \backslash \Gamma$ with $\gamma_{c, d} =  \begin{bmatrix} 0& c^{-1} \\ c & d \end{bmatrix}$, $c \in \Fq^\times$, $d \in \Fq$:
\[
\begin{aligned}
	S_2 & = \sum_{c \in \Fq^\times, d \in \Fq} \frac{(-1)^m}{(c\xi_N + d)^k} G_n \left ( u \left ( \frac{c^{-1}}{c  \xi_N+ d} \right ) \right) \\
	& =  \sum_{c \in \Fq^\times, d \in \Fq} \frac{(-1)^m}{(c\xi_N + d)^k}  \sum_{i = 0}^s \gg_i u \left ( \frac{c^{-1}}{c  \xi_N+ d} \right )^{n - i(q-1)}.
\end{aligned}
\]
Summing on the $c$'s we get
\[
	S_2 = (-1)^{m + 1} \sum_{d \in \Fq} \frac{1}{(\xi_N + d)^k}  \sum_{i = 0}^s \gg_i u \left ( \frac{1}{  \xi_N+ d} \right )^{n - i(q-1)}.
\]
Consider each $i$. As argued on p.306 of \cite{GvP} one has 
\[
	\sum_{d \in \Fq} \frac{\gg_i}{(\xi_N + d)^k}  u \left ( \frac{1}{  \xi_N+ d} \right )^{n - i(q-1)}  = \delta_i + \sum_{d \in \Fq} \frac{\gg_i}{(\xi_N + d)^k}  \left ( \frac{\xi_N + d}{\ppi} \right  )^{n - i(q-1)} ,
\]
with $|\delta_i| <  |\gg_i| |\ppi^{-1}|^{n - i (q-1)} |\xi_N|^{k - n + i (q-1)}$. Since $|d| < |\xi_N|$ we also get
\[
	\sum_{d \in \Fq} \frac{\gg_i}{(\xi_N + d)^k}  u \left ( \frac{1}{  \xi_N+ d} \right )^{n - i(q-1)}  = \delta_i' + \sum_{d \in \Fq} \frac{\gg_i}{(\xi_N)^k}  \left ( \frac{\xi_N}{\ppi} \right  )^{n - i(q-1)},
\]
with $|\delta_i'| <  |\gg_i| |\ppi^{-1}|^{n - i (q-1)} |\xi_N|^{k - n + i (q-1)}$. Summing on the $d$'s gives $0$, so we have
\begin{equation} \label{s2}
	|S_2| \leq \max_{0 \leq i \leq s} | \delta_i' | < \max_{0 \leq i \leq s} |\gg_i| \frac{1}{|\ppi|^{n - i (q-1)}} \frac{1}{|\xi_N|^{k -n + i(q-1)}}.
\end{equation}
The reader should compare this with \eqref{s1}. If $\frac{k}{q+1} \geq n$, then for all $0 \leq i \leq s$ we have $\frac{k}{q + 1} \geq n - i (q-1)$. Consequently, for all $0 \leq i \leq s$ we get $|\xi_N^{-1}|^{q(n - i (q-1))} \geq |\xi_N^{-1}|^{k - n + i (q-1)}$ and hence
\[
	|S_1| > |S_2|.
\]

Finally, we consider the sum corresponding to matrices $\gamma_{c, d}\in H \backslash \Gamma$ with $\gamma_{c, d} =  \begin{bmatrix} a & b\\ c & d \end{bmatrix}$, where $a, b, c, d \in A$ satisfy $a d - bc = 1$,  $\deg_T (c) + \deg_T (d) \geq 1$, $\deg_T (a) < \deg_T (c)$, $\deg_T(b) < \deg_T(d)$:
\[
	S_3 =   \sideset{}{^\ast} \sum_{a, b, c, d \in A} \frac{1}{(c \xi_N + d)^k} G_n \left ( u \left ( \frac{a \xi_N + b}{c \xi_N + d} \right) \right).
\]
For such $a, b, c, d \in A$, we have 
\[
	\left | \frac{a \xi_N + b}{c \xi_N + d} \right | < 1
\]
and therefore
\[
	\left | u \left ( \frac{a \xi_N + b}{c \xi_N + d} \right) \right| = \frac{1}{|\ppi|} \left | \frac{c \xi_N + d}{a \xi_N + b} \right |.
\]
So each term in the sum defining $S_3$ has absolute value
\[
	\begin{aligned}
	&  \leq \max_{0 \leq i \leq s} \frac{|\gg_i|}{|\ppi|^{n + i (q-1)}} \frac{1}{|c \xi_N + d|^{k - n + i (q-1)} |a \xi_N + b|^{n + i (q-1)}} \\
	& < \max_{0 \leq i \leq s} \frac{|\gg_i|}{|\ppi|^{n + i (q-1)}} \frac{1}{|\xi_N|^{k-n + i(q-1)}} \qquad \text{(compare with \eqref{s2})} \\
	& \leq |S_1|.
	\end{aligned}
\]
We have $P_{k, n} (\xi_N) = S_1 + S_2 + S_3$ with $|S_1| > |S_2|$, $|S_1| > |S_3|$ and therefore $P_{k, n} (\xi_N) \neq 0$.
\end{proof}


We will show that if the conditions of Remark~\ref{rmk_imp} (and consequently hyperdifferentiation preserves modularity), then the hyperderivatives of $P_{k, 1}$ are again Drinfeld-Poincar\'{e} series. First, we need the following lemma.

\begin{lemma} \label{lem1} For $w_1, w_2, w_3 \geq 0$, we have
\[
	\sum_{r = 0}^{w_1} \binom{w_2 + r}{r} \binom{w_3 - r}{w_1 - r} = \binom{w_2 + w_3 + 1}{w_1}.
\]
\end{lemma}

\begin{proof}
We use induction on $w_1 + w_2 + w_3$ and Pascal's identity
\[
	\binom{M}{N} = \binom{M-1}{N-1} + \binom{M-1}{N}, \qquad M \geq 0, 0 \leq N \leq M.
\]
The result is clear for $w_1 + w_2 + w_3 = 0$. Assume that the result holds for $w_1 + w_2 + w_3$, we consider three cases. 

First, compute by using Pascal's identity and the induction hypothesis:
\[
	\sum_{r = 0}^{w_1 + 1} \binom{w_2 + r}{r} \binom{w_3 - r}{w_1 + 1 - r}  = \sum_{r = 0}^{w_1 + 1} \binom{w_2 + r}{r} \left ( \binom{w_3 - r -1}{w_1 - r} + \binom{w_3 - r - 1}{w_1 + 1 - r} \right) 
\]
\[
\begin{aligned}
	&  = \sum_{r = 0}^{w_1 + 1} \binom{w_2 + r}{r} \binom{w_3 - r -1}{w_1 - r} + \sum_{r = 0}^{w_1 +1} \binom{w_2 + r}{r} \binom{w_3 - 1 - r}{w_1 + 1 - r} \\
	&  = \sum_{r = 0}^{w_1} \binom{w_2 + r}{r} \binom{w_3 - r -1}{w_1 - r} + \sum_{r = 0}^{w_1 +1} \binom{w_2 + r}{r} \binom{w_3 - 1 - r}{w_1 + 1 - r} \\
	& = \binom{w_2 + w_3}{w_1} + \binom{w_2 + w_3}{w_1+1}  = \binom{w_2 + w_3+ 1}{w_1+1}.
\end{aligned}
\]
Similarly, 
\[
	\sum_{r = 0}^{w_1}  \binom{w_2 +1+ r}{r} \binom{w_3 - r}{w_1 - r}  = \sum_{r = 0}^{w_1} \left ( \binom{w_2 + r}{r} + \binom{w_2 + r}{r-1} \right ) \binom{w_3 - r}{w_1 - r} 
\]
\[
\begin{aligned}
	& =  \sum_{r = 0}^{w_1}  \binom{w_2 + r}{r} \binom{w_3 - r}{w_1 - r}  + \sum_{r = 0}^{w_1} \binom{w_2 + r}{r-1}  \binom{w_3 - r}{w_1 - r} \\
	& = \sum_{r = 0}^{w_1}  \binom{w_2 + r}{r} \binom{w_3 - r}{w_1 - r} + \sum_{r = 0}^{w_1 - 1} \binom{w_2 + r+1}{r}  \binom{w_3 - r -1 }{w_1 - r-1} \\
	& = \binom{w_2 + w_3 + 1}{w_1} + \binom{w_2 + w_3 + 1}{w_1 - 1} = \binom{w_2 + w_3 + 2}{w_1}.
	\end{aligned}
\]
And finally,
\[
	\sum_{r = 0}^{w_1} \binom{w_2 + r}{r} \binom{w_3 + 1 -r}{w_1 - r}  = \sum_{r = 0}^{w_1} \binom{w_2 + r}{r} \left ( \binom{w_3 - r}{w_1 - r - 1} + \binom{w_3 - r}{w_1 - r} \right) 
\]
\[
\begin{aligned}
	& = \binom{w_2 + w_3 + 1}{w_1} + \binom{w_2 + w_3 + 1}{w_1 - 1} = \binom{w_2 + w_3 + 2}{w_1}.
	\end{aligned}
\]
\end{proof}

\begin{theorem} \label{Poin_thm}
If $k$ and $n$ are positive integers such that
\[
      \binom{k+n-1}{j} \equiv 0 \bmod p, \qquad \forall j, 1 \leq j \leq n,
\]
 then
\[
	\DDd_n P_{k, 1} = P_{k + 2n, 1 + n}.
\]
\end{theorem}

\begin{proof} For clarity, we write $f \circ \gamma$ for the function $z \mapsto f(\gamma z)$. In this notation
\[
	P_{k, 1}(z) = \sum_{\gamma \in H \backslash \Gamma} \frac{\det \gamma}{\jg^k} (u \circ \gamma)(z).
\]

Note that
\begin{equation} \label{unaltered}
\Dunalt_n \left ( \frac{1}{\jg^w} \right) =  \binom{w + n -1}{n} (-\Xg)^n \frac{1}{\jg^w},
\end{equation}
and recall \cite[(20)]{BoPe} that:
\begin{equation} \label{main_eq}
\left ( \Dunalt_n (f \circ \gamma) \right) (z) = (-\Xg)^n \sum_{i = 1}^n  \binom{n-1}{n-i} (-\Xg)^{-i} \frac{(\det \gamma)^i}{\jg^{2i}} \left ( \Dunalt_i f \right) (\gamma z). 
\end{equation}

By the Leibniz rule \eqref{leibniz} and formula \eqref{unaltered}:
\[
\begin{aligned}
	\Dunalt_n  \left ( \frac{\det \gamma}{\jg^k}  (u \circ \gamma)(z) \right)  & = \sum_{r = 0}^n \det \gamma \Dunalt_r \left ( \frac{1}{\jg^k} \right) \Dunalt_{n-r} (u \circ \gamma) (z) \\
	& =  \sum_{r = 0}^n \det \gamma \binom{k+r-1}{r} \frac{(-\Xg)^r}{\jg^k} \Dunalt_{n-r} (u \circ \gamma)(z).
\end{aligned}
\] 
Formula \eqref{main_eq} shows that this is equal to 
\[
	 \sum_{r = 0}^n \sum_{i = 1}^{n-r} \binom{k+r-1}{r} \binom{n-r-1}{n-r-i} (-\Xg)^{n-i} \frac{\det \gamma^{1+i}}{\jg^{k + 2i}} \left (\Dunalt_{i} u \right) (\gamma z).
\]
We can change the order of summation and rearrange to get
\[
	\sum_{i = 0}^n (-\Xg)^{n-i} \frac{\det \gamma^{1+i}}{\jg^{k + 2i}} \left ( \Dunalt_{i} u \right) (\gamma z) \sum_{r = 0}^{n-i}  \binom{k+r-1}{r} \binom{n-r-1}{n-r-i}.
\]
By taking $w_1 = n-i$, $w_2 = k - 1$, $w_3 = n-1$ in Lemma~\ref{lem1} we get 
\[
	 \sum_{r = 0}^{n-i}  \binom{k+r-1}{r} \binom{n-r-1}{n-r-i} = \binom{k+n-1}{n-i}.
\]
Therefore $\Dunalt_n  \left ( \frac{\det \gamma}{\jg^k}  (u \circ \gamma)(z) \right)$ equals
\[
	\sum_{i = 0}^n \binom{k+n-1}{n-i} (-\Xg)^{n-i} \frac{\det \gamma^{1+i}}{\jg^{k + 2i}} \left ( \Dunalt_{i} u \right) (\gamma z).
\]

In particular, if
\[
  \binom{k + n -1}{j} \equiv 0 \bmod p, \qquad \forall j, \ 1 \leq j \leq n,
\]
then
\begin{equation} \label{better}
 \DDd_n \left ( \frac{\det \gamma}{\jg^k}  (u \circ \gamma)(z) \right) = \frac{\det \gamma^{1+n}}{\jg^{k + 2n}} (\DDd_{n} u) (\gamma z) = \frac{\det \gamma^{1+n}}{\jg^{k + 2n}} G_{n+1} (u(\gamma z)).
\end{equation}
\end{proof}

Combining this with Theorem~\ref{main}, shows that:

\begin{coro}
If $f_s = -P_{2 + s(q-1), 1}$, then for any $k$ and $n$ such that $k - 2 n$ is a positive multiple of $(q-1)$ and $n \leq p^{\val_p (k - n)}$ we have
\[
	f_{k, n} = -P_{k, n}.
\] 
\end{coro}

\begin{proof} For such $k$ and $n$, let $s = \frac{k-2n}{q-1}$ as in Theorem~\ref{main}. Then
\[
\begin{aligned}
	f_{k, n} = \DDd_{n-1} f_s = - \DDd_{n-1} P_{2 + s(q-1), 1}  & =  -P_{k, n}.
\end{aligned}	
\]	
\end{proof}


\begin{rmk}
The equality $P_{k, n} = -f_{k, n}$, which remains to be shown, would be quite remarkable and would show that the Drinfeld-Poincar\'{e} series  $P_{k, n}$ are eigenforms with eigensystems $\{ \wp^n \}$. In addition, this equality shows that one can compute the expansion at infinity of $P_{k, n}$ (a highly non-trivial task, see \cite[Sec.~9]{Gek}). Unfortunately, at present the only equality between single-cuspidal forms with $A$-expansions and Drinfeld-Poincar\'{e} series that we can prove is $h = -P_{q+1, 1}$ (see \cite[9.1]{Gek} and note that there Gekeler writes $h$ for what in our notation is $-h$). Therefore, we have to content ourselves with the equalities for $\DDd_{n-1} h = -P_{k, n}$ when $\DDd_{n-1} h$ is modular. For $q = 3$ for instance, we have:
\[
\begin{aligned}
P_{16, 7}  & = - \sum_{a \in A_+} a^{9} G_7 (u_a)  = - \frac{1}{T^6 + T^4 + T^2} u^3 - \frac{1}{T^3 - T} u^5 - u^7 + O(u^9). \\
\end{aligned}
\]
\end{rmk}

\bibliographystyle{alpha}
\bibliography{references}

\end{document}